\documentclass[10pt, fleq]{amsart}
\usepackage{amssymb,amsmath,latexsym,amsbsy,color,enumerate}
\numberwithin{equation}{section}
\def\NN{\mbox{$I\hspace{-.06in}N$}}

\def\CC{\mbox{$C\hspace{-.11in}\protect\raisebox{.5ex}{\tiny$/$}
\hspace{.06in}$}}

\def\h{\hspace*{.24in}}

\newtheorem{theorem}{Theorem}
\newtheorem{definition}[theorem]{Definition}
\newtheorem{lemma}{Lemma}

\newtheorem{proposition}{Proposition}
\newtheorem{corollary}{Corollary}

\begin{document}
 \title{Nearly-optimal estimates for the stability problem in Hardy spaces}
    \author{Dang Duc Trong}
    \address{Department of Mathematics, Hochiminh City National University, 227 Nguyen Van Cu, Q5, HoChiMinh City, Vietnam}
    \email{ddtrong@mathdep.hcmuns.edu.vn}
\author{Tuyen Trung Truong}
    \address{Department of mathematics, Indiana University Bloomington, IN 47405 USA}
 \email{truongt@indiana.edu}
    \date{\today}
\begin{abstract}
We continue the work of \cite{TLNT}. Let $E$ be a non-Blaschke subset of the unit disc $\mathbb{D}$ of the complex plane $\mathbb{C}$. Fixed $1\leq p\leq
\infty$, let $H^p(\mathbb{D})$ be the Hardy space of holomorphic functions in the disk whose boundary value function is in $L^p(\partial \mathbb{D})$.
Fixed $0<R<1$. For $\epsilon >0$ define
\begin{eqnarray*}
C_p(\varepsilon, R) = \sup \{\sup_{|z| \leq R}|g(z)|: g\in H^p,\, \|g\|_p\leq 1,\,|g(\zeta)| \leq \varepsilon\;\forall \zeta\in E\}.
\end{eqnarray*}
In this paper we find upper and lower bounds for $C_p(\epsilon ,R)$ when $\epsilon$ is small for any non-Blaschke set $E$. The bounds are nearly-optimal
for many such sets $E$, including sets contained in a compact subset of $\mathbb{D}$ and sets contained in a finite union of Stolz angles.
\end{abstract}
\maketitle
\section{Introduction}
This work is a continuation of \cite{TLNT}. The purpose of this paper is to find good estimates for the stability problem of approximating analytic
functions in Hardy spaces.

Let $E$ be a subset of the unit disc $\mathbb{D}$ of the complex plane $\CC$. To avoid trivial counter-examples, we assume throughout this paper that $E$
is non-Balschke, that is
\begin{description}
\item[(B)] $E$ contains a non-Blaschke sequence $(z_j)$, that is, a sequence satisfying the condition
\[
\displaystyle\sum_{j=1}^\infty(1-|z_j|) = \infty.
\]
\end{description}
Also without loss of generality, we assume throughout that $E$ is relatively closed in $\mathbb{D}$, that is if $\overline{E}$ is the closure of $E$ in
the usual topology in $\mathbb{C}$ then $\overline{E}\cap \mathbb{D}=E$.

Fixed $1\leq p\leq \infty$, recall that the Hardy space $H^p(\mathbb{D})$ is the space of all holomorphic functions $g$ on $\mathbb{D}$ for which
$\|g\|_p$ $<$ $\infty$, where
\begin{align*}
\|g\|_p &= \lim_{r\uparrow 1} \left\{\frac{1}{2\pi}\int_0^{2\pi}|g(re^{i\theta})|^p d\theta\right\}^{1/p}\h(1 \leq p < \infty),\\
\|g\|_{\infty} &= \lim_{r\uparrow 1} \sup_{\theta} |g(re^{i\theta})|.
\end{align*}
For convenience, from now on, we will denote $H^p(\mathbb{D})$ by $H^p$. We define $\mathcal{A}^p$ to be the functions in $H^p$ with norm $1$, that is
\begin{equation}
\mathcal{A}^p=\{f:~f\in H^p, ||f||_{H^p}=1\}.\label{EquationSpaceAp}
\end{equation}
If $f\in \mathcal{A}^p$ it follows that (see Section \ref{SectionInterpolationByBalschkeProducts})
\begin{equation}
|f(z)|\leq \frac{1}{(1-|z|^2)^{1/p}}, \label{EquationEstimateForFunctionsInSpaceAp}
\end{equation}
for all $z\in \mathbb{D}$.

If $f$ is a function in $H^p(\mathbb{D})$ then it is well-known that $f$ can be reconstructed from its values $f(\zeta )$ at points $\zeta\in E$ (see
Theorem \ref{TheoremRecovery}). However, in practice, it is usually the case that we do not know exact values $f(\zeta )$, but only approximate values.
This leads to the stability problem, that of estimating the quantity
\begin{equation}
C_p(E,\varepsilon, R) = \sup \{\sup_{|z| \leq R}|g(z)|: g\in \mathcal{A}^p,|g(\zeta)| \leq \varepsilon\;\forall \zeta\in E\}, \label{CeRdefinition}
\end{equation}
for positive $\varepsilon$ and $R$ in $(0, 1)$. We can also consider the problem of one-point estimation, which is estimating the number
\begin{equation}
C_p(E,\varepsilon, 0) = \sup \{|g(0)|: g\in \mathcal{A}^p,|g(\zeta)| \leq \varepsilon\;\forall \zeta\in E\}. \label{Ce0definition}
\end{equation}
Since $E$ satisfies (B), it is well-known that
\[
\lim_{\varepsilon \rightarrow 0}C_p(\varepsilon, R) = 0.
\]

This problem of estimating $C_p(E,\epsilon ,R)$ was thoroughly explored by many authors. Let us recall some of the results known in literature.

In \cite{LRS}, Lavrent'ev, Romanov and Shishat-skii used a certain characteristic of the projection of $E$ onto the real axis, to show that if $E$
$\subset$ $U$ $=$ $\{z: |z| \leq 1/4\}$ then $C_p(\varepsilon, R)$ $\leq$ $\max\{\varepsilon^{4/25},(6/7)^{n(\varepsilon)}\}$ for all $R$ $\in$
$(0,1/4)$, in which $n(\varepsilon)$ $\rightarrow$ $\infty$ as $\varepsilon$ $\rightarrow$ $0$. This approach is quite interesting in that $E$ could be a
sequence. However in their approach the set $E$ is strictly contained inside $\mathbb{D}$, and only upper bounds are obtained.

In a series of works (\cite{osi}, \cite{osi7}, \cite{osi6}, \cite{osi5}, \cite{osi4}, \cite{osi3} and \cite{osi2}), Osipenko obtained optimal estimates
for some special sets $E$. For example, when $E$ is contained in the real open interval $(-1,1)$ and satisfies some more constrains, he showed that the
optimal value of $C_p(E,\epsilon ,0)$ is obtained at a finite Balschke product $B(z)$ with all zeros in $E$, that is
\begin{eqnarray*}
B(z)=\prod _{j=1}^n\frac{z-z_j}{1-\overline{z}_jz},
\end{eqnarray*}
and here $z_j\in \overline{E}\subset \mathbb{D}$. It is interesting that here the set $E$ needs not to be contained in a compact set of $(-1,1)$.
However, his method seems not applicable to more general sets $E$.

In case $p=\infty$, it is well-known that the set of boundary limit points, or more exactly non-tangential limit points, of $E$ plays an important role
in estimating $C_{\infty}(E,\epsilon ,R)$. Let us first recall the definition of non-tangential limit points $E_0$ of $E$ (see \cite{hay}):
\begin{definition}
For each set $E$ of $\mathbb{D}$, we denote by $E_0$ the set of nontangential limit points of $E$, that is, points $\zeta$ of $\partial \mathbb{D}$ being
such that there exists a sequence $(z_n)$ in $E$ which tends nontangentially to $\zeta$, that is, such that
\begin{eqnarray*}
z_n\rightarrow \zeta,~|z_n-\zeta|=O(1-|z_n|).
\end{eqnarray*}\label{DefinitionNonTangentialLimitPoints}
\end{definition}

Let $m(E_0)$ be the Lebesgue measure of $E_0$ as a subset of $\partial \mathbb{D}$. If $m(E_0)>0$, we can use the harmonic measure $\omega (z)$ of $E_0$
to obtain the following estimate (see the Appendix):
\begin{equation}
\epsilon \leq C_{p}(E,\epsilon ,R)\leq \frac{2^{1/p}}{(1-R^2)^{1/p}}\sup _{|z|\leq R}\epsilon ^{\omega (z)}.
\label{EquationEstimateCaseE0HasPositiveMeasure}
\end{equation}
Hence in case $m(E_0)>0$ we obtain a quasi-polynomial estimate for $C_p(E,\epsilon ,R)$.

The main purpose of this paper is to obtain good upper and lower bounds for $C_p(E,\epsilon ,R)$ for the remaining case when $m(E_0)=0$ in such a way to
extend the above mentioned results of Lavrente's et al. and Osipenko. Our idea consists of two steps:

-Step 1: Use the interpolation by finite Balschke product to reduce estimating $C_p(E,\epsilon ,R)$ to estimating of some expressions depending only on
$\epsilon$ and finite Blaschke products with all zeros in $E$. This step 1 was already done in our previous paper (see Section 3 in \cite{TLNT}), where
an algorithm for choosing the interpolation points was proposed. However that algorithm depends on the ordering of the sequence $(z_k)$, and the method
used there does not allow obtaining lower bounds for $C_{p}(E,\epsilon ,R)$. We propose a better algorithm in Step 2 below, which allows us to obtain
both upper and lower estimates for $C_{p}(E,\epsilon ,R)$, and to obtain nearly-optimal estimates for many sets $E$ (see Corollaries
\ref{CorollaryCaseEIsCompact} and \ref{CorollaryCaseEIsInStolzAngles}).

-Step 2: For any $n\geq 0$, assigns a number $M_n(E)$ using finite weighted-Blaschke products (see Definition \ref{DefinitionWeightedBlaschkeProducts})
to construct set functions for $E$. Then we use these functions $M_n(E)$ to estimate the expressions in Step 1.

Explicitly we fix a bounded holomorphic function $q(z)$ in $\mathbb{D}$ satisfying the following conditions: $q(z)\not= 0$ for all $z\in \mathbb{D}$ and
\begin{equation}
\lim _{z\in E,z\rightarrow \partial\mathbb{D}}q(z)=0.\label{EquationTheFunctionQ}
\end{equation}
The function $q(z)$ mentioned above is provided by the following Theorem by Hayman\cite{hay}:
\begin{theorem}
If the set $E_0$ of nontangential limit points of a set $E$ has positive linear measure and if $f$ is a bounded analytic function satisfying
\begin{eqnarray*}
\lim _{z\in E,~|z|\rightarrow 1}f(z)=0,
\end{eqnarray*}
then $f\equiv 0$. Conversely, if $E_0$ has measure zero, then there exists $f(z)$, such that $0<|f(z)|<1$ in $U$, and satisfying
\begin{eqnarray*}
\lim _{z\in E,~|z|\rightarrow 1}f(z)=0.
\end{eqnarray*}
\label{TheoremHayman}\end{theorem}

Before stating our main results, let us fix some notations.
\begin{definition}
We will use the notation $Z_n=\{z_1,\ldots ,z_n\}$ to denote a tuple of $n$-points $z_1,\ldots ,z_n\in \mathbb{D}$. If $j\in \{1,\ldots ,n\}$ we define
$Z_{n,j}=\{z_1,\ldots ,z_n\}\backslash \{z_j\}$. Define $B(Z_n,z)$ to be the Blaschke product with zeros in $Z_n$:
\begin{eqnarray*}
B(Z_n,z)=\prod _{j=1}^n\frac{z-z_j}{1-\overline{z}_jz}.
\end{eqnarray*}
Similarly define $B(Z_{n,k},z)$ to be the Blaschke product with zeros in $Z_{n,k}$:
\begin{eqnarray*}
B(Z_{n,k},z)=\prod _{1\leq j\leq n, ~j\not= k}\frac{z-z_j}{1-\overline{z}_jz}.
\end{eqnarray*}

For a fixed function $q(z)$, the weighted Blaschke product $B_q(Z_n,z)$ is defined as
\begin{eqnarray*}
B_q(Z_n,z)=q(z)\prod _{j=1}^n\frac{z-z_j}{1-\overline{z}_jz}.
\end{eqnarray*}
\label{DefinitionWeightedBlaschkeProducts}\end{definition}

Let $q(z)$ be a function provided by Theorem \ref{TheoremHayman}. Let us define
\begin{equation}
g(E,\epsilon ,R,q)=\sup \{\sup _{|z|\leq R}|B_q(Z_n;z)|:~n\in \mathbb{N},~Z_n\in E^n,~|B_q(Z_n;\zeta )|\leq \epsilon~\forall \zeta \in
E\}.\label{EquationOfFunctionG}
\end{equation}
(Note that $g(E,\epsilon ,R,q)$ does not depend on $p$.)

\begin{theorem} Let $E\subset \mathbb{D}$ be such that $m(E_0)=0$. Fix $1\leq p\leq \infty$ and $0<R<1$. Let $q(z)$ be a function provided by Theorem \ref{TheoremHayman},
normalized by $||q||_{\infty}=1$  where $||q||_{\infty}$ is its usual sup-norm. Define $g(E,\epsilon ,R,q)$ as in (\ref{EquationOfFunctionG}). Then there
exists $\epsilon _0>0$ depending on $E$ and $q(z)$, and a non-increasing function $\varphi :(0,\epsilon _0)\rightarrow (0,\infty )$ also depending on $E$
and $q(z)$ satisfying
\begin{eqnarray*}
\lim _{\epsilon \rightarrow 0}\varphi (\epsilon )=0,
\end{eqnarray*}
, a constant $K>0$ depending only on $p$ and $R$, and a constant $\alpha >0$ depending only on $R$, such that for all $0<\epsilon <\epsilon _0$ we have
\begin{equation}
g(E,\epsilon ,R,q)\leq C_p(E,\epsilon ,R)\leq K\times |q(0)|^{-\alpha}\times g^{\alpha}(E,\varphi (\epsilon
),R,q).\label{EquationEstimateCaseE0HasZeroMeasure}
\end{equation}
\label{TheoremMain}
\end{theorem}
A class of sets $E$ satisfying the condition $m(E_0)=0$ are those contained in a finite union of Stolz angles, which we recall in the following
\begin{definition}
Let $\zeta \in \partial \mathbb{D}$. A Stolz angle with vertex $\zeta$ is a set of the form
\begin{eqnarray*}
\Omega _{\sigma }(\zeta ):=\{z\in \mathbb{D}:~|1-\overline{z}\zeta|\leq \sigma (1-|z|)\},
\end{eqnarray*}
where $\sigma \geq 1$ is some constant. \label{DefinitionStolzAngle}\end{definition}

The following corollaries can be considered as extensions of above results of Lavrent'ev et al. and Osipenko:
\begin{corollary}
If $E$ is a compact subset in $\mathbb{D}$ then there exist constants $K>0$ and $\epsilon _0>0$ depending only on $p$ and $R$, and there exists a
constant $\alpha >0$ depending only on $R$ such that for all $0<\epsilon <\epsilon _0$, there exists a finite Blaschke product $B(z)$ with all zeros in
$E$ such that
\begin{eqnarray*}
\sup _{|z|\leq R}|B(z)|\leq C_p(E,\epsilon ,R)\leq K\times \sup _{|z|\leq R}|B(z)|^{\alpha}.
\end{eqnarray*}
\label{CorollaryCaseEIsCompact}\end{corollary}
\begin{corollary}
If $E$ is contained in a finite union of Stolz angles then there exist constants $K_{p}$, $\sigma >0$ and $\epsilon _0>0$ depending only on $R$ and the
vertices of these Stolz angles, and there exists a constant $\alpha >0$ depending only on $R$ such that for all $0<\epsilon <\epsilon _0$, there exists a
finite Blaschke product $B(z)$ with all zeros in $E$ such that
\begin{eqnarray*}
\frac{1}{K}\sup _{|z|\leq R}|B(z)|\leq C_p(E,\epsilon ,R)\leq K\times \sup _{|z|\leq R}|B(z)|^{\alpha \sigma}.
\end{eqnarray*}\label{CorollaryCaseEIsInStolzAngles}
\end{corollary}
These results are in fact corollaries of a more general result (see Corollary \ref{CorollaryMnIsSmall}) which needs only the condition that
$M_n(E)\lesssim n^{-\sigma}$ for some constants $\sigma >0$ and all $n\geq 0$.

Let us remark some features of the set functions $M_n(E)$ in Step 2 above. They are analogous to the set functions defined in (weighted) potential theory
for subsets of $\mathbb{C}$ (however there are important differences, see Section \ref{SectionSetFunctions} for more detailed). In fact in case $E$ is
compact in $\mathbb{D}$, we choose $q(z)=1$, and the function $M_n(E)$ is similar to the classical potential theory for the unit disk (see for example
\cite{Tsuji}). In a next paper of the second author, it is shown that by choosing a suitable function $q(z)$ these set functions can be defined for all
subsets $E$ of $\mathbb{D}$ (not only sets $E$ with $m(E_0)=0$ as dealt with in this paper), which give a uniform estimate to a quantity analogous to
$C_p(E,\epsilon ,R)$.

Our approach using interpolation by finite Blaschke products also give a simple and constructive proof to the following result by Danikas\cite{dan} and
Hayman\cite{hay} (see also \cite{lk} for a related result)
\begin{theorem}
Assume that $E$ is a non-Blaschke sequence $(z_j)$. Then there exists a sequence of positive numbers $(\eta _j)$ with
the property that
\begin{eqnarray*}
\lim _{j\rightarrow\infty}\eta _j=0,
\end{eqnarray*}
such that if $f$ is a non-zero bounded analytic funtion on $U$ then
\begin{eqnarray*}
\limsup _{j\rightarrow \infty}\frac{|f(z_j)|}{\eta _j}=\infty .
\end{eqnarray*}
\label{TheoremDanikasAndHayman}
\end{theorem}
This paper is organized as follows. In Section \ref{SectionInterpolationByBalschkeProducts} we recall the formula for interpolation by finite Blaschke
product, some properties of finite Blaschke product, and give a proof of Theorem \ref{TheoremDanikasAndHayman}. In Section \ref{SectionSetFunctions} we
define set functions $M_n(E)$ and other set functions, and the function $\varphi (\epsilon )$ used in Theorem \ref{TheoremMain}. We prove Theorem
\ref{TheoremMain} in Section \ref{SectionProofOfMainTheorem}. We prove Corollaries \ref{CorollaryCaseEIsCompact} and \ref{CorollaryCaseEIsInStolzAngles}
and give some other examples in Section \ref{SectionCorollaryAndExample}. In Section \ref{SectionOnePointEstimate} we prove the similar results for the
one-point estimates of $C_p(E,\epsilon ,0)$. In the Appendix we give the proof of (\ref{EquationEstimateCaseE0HasPositiveMeasure}) for the case when
$m(E_0)>0$.

{\bf Acknowledgment} The second author would like to thank Professor Norman Levenberg for showing us the analog between our set functions in Section 3
and those of the (weighted) potential theory for a subset of $\mathbb{C}$, and for suggesting about using harmonic measures in proving
(\ref{EquationEstimateCaseE0HasPositiveMeasure}). He also would like to thank Professor Yuril Lyubarskii for helpful comments, in particular for showing
us the proof of (\ref{EquationEstimateCaseE0HasPositiveMeasure}) that we include in the Appendix.
\section{Interpolation by finite Blashcke products}\label{SectionInterpolationByBalschkeProducts}
We use the notations in Definition \ref{DefinitionWeightedBlaschkeProducts}.

The following result give an interpolation using Blaschke products for functions in $H^p$:
\begin{theorem}
If $Z_n$ = $(z_1, z_2, \ldots, z_n)$ is a sequence of $n$ distinct points in $\mathbb{D}$ then, for all $f$ in $H^p$ and $z$ in $\mathbb{D}$, the
following inequality holds:
\begin{equation}
\left|f(z) - \sum_{k = 1}^n c_k(Z_n,z)f(z_k)\right| \leq \frac{\|f\|_p}{(1 - |z|^2)^{\frac{1}{p}}}|B(Z_n,z)|,
\label{rtieqn}
\end{equation}
where
\begin{equation}
c_{p,k}(Z_n,z) = \frac{1 - |z_k|^2}{1-\overline{z_k}z}\left(\frac{1 - \overline{z}z_k}{1 - |z|^2}\right)^{\frac{2 -
p}{p}}\frac{B(Z_{n,k},z)}{B(Z_{n,k},z_k)}. \label{rtckexpr}
\end{equation}
\label{TheoremRecovery}
\end{theorem}

The reader is referred to \cite{osi} or \cite{TLNT} for proof of this Theorem.

We need some estimates of $B(Z_n, z)$ and $B(Z_{n,k}, z)$, whose proofs are straightforward.

\begin{proposition}
\begin{align*}
|B(Z_n, z)| &\leq \exp\left(-\frac{1 - |z|^2}{4}\sum_{j = 1}^n (1 - |z_j|)\right),\\
|B_k(Z_n, z)| &\leq 2\exp\left(-\frac{1 - |z|^2}{4}\sum_{j = 1}^n (1 - |z_j|)\right),
\end{align*}
for $z$ in $\overline{\mathbb{D}} $ and $Z_n$ in $\overline{\mathbb{D}}^n$. \label{BlaschkeProductsEstimationsProp}
\end{proposition}

Now we prove Theorem \ref{TheoremDanikasAndHayman}
\begin{proof}{ (of Theorem \ref{TheoremDanikasAndHayman})}  From properties of $E$, we can choose a sequence of integers $n_1<n_2<...<n_k<...$ such that
\begin{eqnarray*}
\sum _{j=n_k}^{n_{k+1}-1}(1-|z_j|)\geq k.
\end{eqnarray*}

It follows that $m_k=n_{k+1}-n_{k}\geq k$. We denote $Z_{(k)}=\{z_{n_k},z_{n_{k}+1},...,z_{n_{k+1}-1}\}$ (this notation is used only in this proof and
just for the sake of simplicity). Then if $n_k\leq j<n_{k+1}$ we define as before $Z_{(k),j}=\{z_{n_k},z_{n_{k}+1},...,z_{n_{k+1}-1}\}\backslash
\{z_j\}$. We define the sequence $\eta _j$ as follows
\begin{eqnarray*}
\eta _j=\frac{|B(Z_{(k),j},z_j)|}{m(k)},
\end{eqnarray*}
if $n_k\leq j<n_{k+1}$. It is easy to see that $\eta _j\rightarrow 0$ as $j\rightarrow\infty$.

Now assume that $f$ is a bounded analytic function satisfying
\begin{eqnarray*}
\limsup _{j\rightarrow \infty}\frac{|f(z_j)|}{\eta _j}<\infty ,
\end{eqnarray*}
we will show that $f\equiv 0$. Indeed, fixed $z\in U$ with $|z|\leq 1/2$. Applying Theorem \ref{TheoremRecovery} for $Z_{(k)}$ and using Proposition
\ref{BlaschkeProductsEstimationsProp} we have
\begin{eqnarray*}
|f(z)|&\leq&C(1+\sum _{j=n_{k}}^{n_{k+1}-1}\frac{|f(z_j)|}{m_k\eta _j})\max _{j=n_k,...,n_{k+1}-1}|B(Z_{(k),j},z)|\\
&\leq&C\exp \{(-k+1)/4\},
\end{eqnarray*}
for all $k$. So $f(z)=0$ for all $|z|\leq 1/2$. Hence $f\equiv 0$.
\end{proof}
We conclude this section by some more estimates on weighted Blashcke products used later on.
\begin{lemma}
If $R$ is a real number in $(0,1)$, then there exists a positive number $\alpha$ depending only on $R$ such that for all $r$ in $[0,1]$, the inequality
underneath holds,
\begin{equation}
\max\{R^\alpha, r^\alpha\} \geq \frac{R + r}{1 + Rr}. \label{EquationRr}
\end{equation}
\label{LemmaExistenceOfAlpha}
\end{lemma}
\begin{proof}
First, we consider the case $r$ $\leq$ $R$. We have $\max\{R^\alpha, r^\alpha\}$ $=$ $R^\alpha$ and $$\displaystyle\frac{R + r}{1 + Rr}\leq \frac{2R}{1 +
R^2}.$$ Thus, if this is the case, we must choose $\alpha$ in such a way that $$0<\alpha\leq\frac{\ln(2R) - \ln(1 + R^2)}{\ln R}.$$

Finally, we consider the case $r>R$. The inequality (\ref{EquationRr}) is now equivalent to $$\displaystyle\frac{r^\alpha - r}{1 - r^{\alpha + 1}}\geq
R.$$ We will show that the function $$f(r)=\frac{r^\alpha - r}{1 - r^{\alpha + 1}},~R\leq r\leq 1$$ attains its absolute minimum at $R$. We have $$f'(r)
= \frac{r^{2\alpha} - \alpha r^{\alpha + 1} + \alpha r^{\alpha - 1} - 1}{(1 - r^{\alpha + 1})^2}.$$ Define $$g(r)=r^{2\alpha} - \alpha r^{\alpha + 1} +
\alpha r^{\alpha - 1} - 1,~R\leq r\leq 1,$$ then $$g'(r)=2\alpha r^{2\alpha - 1} - \alpha(1 + \alpha)r^\alpha - \alpha(1 - \alpha)r^{\alpha - 2}.$$ By
Holder inequality
\begin{eqnarray*}
\frac{x^p}{p}+\frac{y^q}{q}\geq xy, ~x,y\geq 0,~\frac{1}{p}+\frac{1}{q}=1,
\end{eqnarray*}
applied to
\begin{eqnarray*}
x&=&r^{1-\alpha},\\
y&=&r^{-(1+\alpha)},\\
p&=&\frac{2}{1+\alpha},\\
q&=&\frac{2}{1-\alpha},
\end{eqnarray*}
one has $$(1 + \alpha)r^{1 - \alpha} + (1 - \alpha)r^{-(1 + \alpha)}\geq 2$$ if $0<r<1$, $0<\alpha <1$. This shows that $g'(r)$ $\leq$ $0$ and thus
$g(r)$ $\geq$ $g(1)$ $=$ $0$. As a consequence, $f(r)$ is non-decreasing, hence when $1\geq r\geq R$ we have
$$f(r)\geq f(R)=\displaystyle\frac{R^\alpha - R}{1 - R^{\alpha + 1}}.$$

Therefore, the proof of the lemma is complete once we can show that for sufficiently small $\alpha$ the inequality $f(R)\geq R$ holds. Indeed, this is
equivalent to $R^{\alpha - 1} + R^{\alpha + 1}\geq 2$. Since $0<R<1$, it follows that $R^{-1} + R^1>2$. Hence, choosing $\alpha$ small enough leads to
the desired result.
\end{proof}
\begin{lemma}
Fix $1>R>0$. Then there exists a constant $\alpha >0$ depending only on $R$ such that for all holomorphic function $q(z)$ and $Z_n=\{z_1,\ldots ,z_n\}\in
\mathbb{D}^n$, and $B_q(Z_n,z)$ the weighted Blaschke product with zeros in $Z_n$ we have
\begin{eqnarray*}
\sup _{|z|\leq R}|B_q(Z_n,z)|^{\alpha}\geq |q(0)|^{\alpha}\prod _{j=1}^n\frac{R+|z_j|}{1+R|z_j|}.
\end{eqnarray*}
In particular, if $q(z)$ is as in Theorem \ref{TheoremMain}, for any $1\leq k\leq n$ we have
\begin{eqnarray*}
\frac{1}{R}|q(0)|^{-\alpha}\sup _{|z|\leq R}|B_q(Z_n,z)|^{\alpha}\geq \sup _{|z|\leq R}|B(Z_{n,k},z)|.
\end{eqnarray*}
\label{LemmaSomeEstimatesForBlaschkeProducts}\end{lemma}
\begin{proof}
By Jensen's formula (see \cite{rud})
\begin{eqnarray*}
\sup _{|z|\leq R}|B(Z_n,z)|\geq |q(0)|\prod _{j=1}^n\max\{R,|z_j|\}.
\end{eqnarray*}
Choose $\alpha$ as in Lemma \ref{LemmaExistenceOfAlpha} we have the conclusion of Lemma \ref{LemmaSomeEstimatesForBlaschkeProducts}.
\end{proof}
\section{Some set functions}\label{SectionSetFunctions}
We use notations in Sections \ref{DefinitionNonTangentialLimitPoints} and \ref{DefinitionWeightedBlaschkeProducts}. Assume throughout this Section that
$E$ is a relative closed subset in $\mathbb{D}$ having infinitely many points, whose non-tangential limit points $E_0$ has Lebesgue measure zero:
$m(E_0)=0$. Fixed $q(z)$ a function provided by Theorem \ref{TheoremHayman}, normalized by $||q||_{\infty}=1$. (If $E$ is compact in $\mathbb{D}$ we take
$q(z)\equiv 1$).

Let us introduce some definitions.
\begin{definition}
Let $Z_n$ $=$ $(z_1, z_2, \ldots, z_n)$ $\in$ $\overline{\mathbb{D}}^n$. For all $0\leq j\leq n$ define $Z_{j}=\{z_1,\ldots ,z_j\}$, in particular
$Z_0:=\emptyset$. Put
\begin{align} V(Z_n)
    &= \prod_{1 \leq j \leq n}|B_q(Z_{j-1}, z_j)|,\\
\mu(z_1, z_2, \ldots, z_n)
    &= \sum_{1 \leq j \leq n} \frac{1}{|B_j(Z_n, z_j)|},\\
M(z_1, z_2, \ldots, z_n)
    &= \sup_{z \in E}|B_q(Z_n, z)|.
\end{align}\label{DefinitionSetFunctions}
\end{definition}
The function $V(Z_n)$ in the above definition can be more explicitly written as
\begin{equation}
V(Z_n)=\prod _{j=1}^n|q(z_j)|\prod _{1\leq j<k\leq n}|\frac{z_j-z_k}{1-\overline{z_j}{z_k}}|.\label{EquationOfFunctionV}
\end{equation}

\begin{definition}
Let $E$ be a subset of $\overline{\mathbb{D}}$ which contains infinitively many points. Put
\begin{align}
V_n(E)
    &= \sup_{Z_n \in E^n} V(Z_n),\\
\mu_n(E)
    &= \inf_{Z_n \in E^n,~V(Z_n)=V_n(E)} \mu(Z_n),\\
M_n(E)
    &= \inf_{Z_n \in E^n,~V(Z_n)=V_n(E)} M(Z_n).
\end{align}\label{DefinitionOptimalSetFunctions}
\end{definition}

The set functions defined above are analog to the set functions defined in (weighted) potential theory for subsets of $\mathbb{C}$ (see for example
Section 5.5 in \cite{ransford}). The sequence $Z_n\in \overline{E}^n$ for which $V_n(E)=V(Z_n)$ are analog to the Fekete points. In case $q(z)\not\equiv
1$, $V_n(E)^{2/n(n-1)}$ is an analog of the $n$-th diameter. However, for $q(z)\not\equiv 1$, $V_n$ has no analog in the weighted potential theory for
$\mathbb{C}$. This is because the function $q(z)$ occurs in $V(Z_n)$ only $n$ times instead of $n(n-1)/2$ times.

\begin{lemma}
If $V_n(E)$ $=$ $V(z_1, z_2, \ldots, z_n)$ then $z_j$ $\in$ $E$ for all $j$ $=$ $1$, $2$, \ldots, $n$, and $|B_{q}(Z_{n,j}, z_j)|$ $=$ $\displaystyle
\sup_{z \in E}|B_{q}(Z_{n,j}, z)|=M(Z_{n,j})$. \label{LemmaFeketePoints}\end{lemma}
\begin{proof}

Since $E$ has infinitely many points, we have that $V_n(E)>0$. Hence since
\begin{eqnarray*}
\lim _{z\in E,|z|\rightarrow 1}|q(z)|=0,\end{eqnarray*} and since $|B(z)|\leq 1$ for any Blaschke product, it follows that $z_j$ $\in$ $E$ for all $j$.

From the definition of $V(Z_n)$ we see that
\begin{eqnarray*}
0<V(Z_n)=V(Z_{n,j})\times |B_q(Z_{n,j},z_j)|.
\end{eqnarray*}
Since $V(Z_n)=V(E_n)$ it follows that $|B_q(Z_{n,j},z_j)|=M(Z_{n,j})$.
\end{proof}
\begin{proposition}
Let $z_1$, $z_2$, \ldots, $z_n$ and $\zeta_1$, $\zeta_2$, \ldots, $\zeta_{n+1}$ be points in $\overline E$ such that $ V(z_1, z_2, \ldots, z_n)$ $=$ $ V_n$ and $ V(\zeta_1, \zeta_2, \ldots, \zeta_{n+1})$ $=$ $ V_{n+1}$, then
\[
\mu(\zeta_1, \zeta_2, \ldots, \zeta_{n+1}) M(z_1, z_2, \ldots, z_n) \leq (n+1).
\]
\label{Mu_TildeMRelationProp}
\end{proposition}

\begin{proof}{}
If $z_0$ is the point in $\overline E$ such that $|B_q(Z_n, z_0)|$ $=$ $\displaystyle\prod_{1 \leq j \leq n}d(z_0, z_j) |q(z)|$ $=$ $ M(z_1, z_2, \ldots,
z_n)$, then $ M(z_1, z_2, \ldots, z_n)  V(z_1, z_2, \ldots, z_n)$ $=$ $ V(z_0, z_1, z_2, \ldots, z_n)$ $\leq$ $ V_{n+1}$ (see Lemma
\ref{LemmaFeketePoints}).

Therefore, for $k$ $=$ $1$, $2$, \ldots, $n+1$, we have
\begin{align*}
 M(z_1, z_2, \ldots, z_n)
    &\leq \frac{ V_{n+1}}{ V(z_1, z_2, \ldots, z_n)} \leq \frac{ V(\zeta_1, \zeta_2, \ldots, \zeta_{n+1})}{ V(\zeta_1, \ldots, \zeta_{k-1}, \zeta_{k+1}, \ldots, \zeta_n)}\\
    &=|q(\zeta_k)| \prod_{1 \leq j \neq k \leq n+1}|\frac{\zeta _j-\zeta _k}{1-\overline{\zeta}_j\zeta_k}| \leq 1.
\end{align*}
It follows that
\[
\mu(\zeta_1, \zeta_2, \ldots, \zeta_{n+1}) \leq \frac{(n+1)}{ M(z_1, z_2, \ldots, z_n)}.
\]
This proves the proposition.
\end{proof}

\begin{proposition}
$\displaystyle \lim_{n \rightarrow \infty}  V_n^{1/n}$ $=$ $\displaystyle \lim_{n \rightarrow \infty}  M_n$ $=$ $0$.
\label{ConvergenceOfTildeV_TildeMProp}
\end{proposition}

\begin{proof}{}
If $E$ is compact in $\mathbb{D}$ then there exists $1>r>0$ such that for all $z\in E$ we have $|z|\leq r$. Hence
\begin{eqnarray*}
V_n^{1/n}\leq (\frac{2r}{1+r^2})^{(n-1)/2}\rightarrow 0
\end{eqnarray*}
as $n\rightarrow 0$.

We now consider the case in which $\overline {E}\cap \partial{\mathbb{D}} \not=\emptyset$.

Fix a number $\delta$ $>$ $0$. By properties of $q(z)$ (see \cite{hay}), it follows that there exist an $r$ $<$ $1$ such that $|z|$ $<$ $r$ whenever $z$
$\in$ $\overline E$ and $q(z)$ $>$ $\delta$. For each n, we rearrange $z_1$, $z_2$, \ldots, $z_n$ so that: there is a constant $k_n$ for which
$|q(z_j)|\leq \delta$ for $1\leq j\leq k_n$, and $|z_j|\leq r$ for $k_{n}+1\leq j\leq n$. We have
\begin{align*}
 V_n
    &= \prod_{1 \leq j < l \leq n}d(z_j, z_l) \prod_{1 \leq j \leq n}|q(z_j)|\\
    &\leq\prod_{k_n + 1 \leq j < l \leq n}d(z_j, z_l) \prod_{1 \leq j \leq k_n}|q(z_j)| \leq \eta^{(n - k_n)(n - k_n - 1)/2} \delta^{k_n},
\end{align*}
where $\eta$ $=$ $\frac{2r}{1 + r^2}$. It follows that $ V_n^{1/n}$ $\leq$ $\eta^{(n - k_n)(n - k_n - 1)/2n}\delta^{k_n/n}$. From this, we see that, if
$k_n/n$ $\geq$ $1/3$, then $ V_n^{1/n}$ $\leq$ $\delta^{1/3}$, and if  $k_n/n$ $<$ $1/3$, then $ V_n^{1/n}$ $\leq$ $\eta^{n/9}$. Hence
\[
\limsup_{n \rightarrow \infty} V_n^{1/n} \leq \limsup_{n \rightarrow \infty} \max\{\delta^{1/3},\eta^{n/9}\} = \delta^{1/3}.
\]
Since $\delta$ can be chosen arbitrarily, we deduce $\displaystyle\lim_{n \rightarrow \infty} V_n^{1/n}$ $=$ $0$.

To prove the second part of Proposition \ref{ConvergenceOfTildeV_TildeMProp}, we choose $Z_n=\{z_1,\ldots ,z_n\}\in E$ so that $V_n(E)=V(Z_n)$. Noting
that $|B_q(Z_n,z)|$ $\leq$ $|B_{q}(Z_{n,j},z)|$ and $|B_{q}(Z_{n,j},z_j)|$ $\leq$ $|B_{q}(Z_{j-1},z_j)|$ for all $j$ = $1$, $2$, \ldots, $n$, using Lemma
\ref{LemmaFeketePoints} we have
\begin{align*}
 M(Z_n)
    &=  M(z_1, z_2, \ldots, z_n) = \sup_{z \in E}|B_q(Z_n,z)| \leq \left(\prod_{1 \leq j \leq n} \sup_{z \in E}|B_{q}(Z_{n,j}, z)|\right)^{1/n}\\
    &= \left(\prod_{1 \leq j \leq n} |B_{q}(Z_{n,j}, z_j)|\right)^{1/n} \leq \left(\prod_{1 \leq j \leq n} |B_q(Z_{j-1}, z_j)|\right)^{1/n} =  V_n(E)^{1/n}.
\end{align*}
Taking supremum on all $Z_n$ with $V(Z_n)=V_n(E)$ we obtain
\begin{eqnarray*}
M_n(E)\leq V_n(E)^{1/n}.
\end{eqnarray*}
This leads to the convergence of $M_n$ to $0$.
\end{proof}

Now we define the function $\varphi (\epsilon )$ in Theorem \ref{TheoremMain}. Applying Proposition \ref{Mu_TildeMRelationProp}, there exists a
continuous function $h:[1,\infty )\rightarrow (0,\infty )$ such that $h$ is non-increasing, $\lim _{x\rightarrow\infty}h(x)=0$ and $M_n\leq h(n)$for all
$n\in\NN$. We can define such an $h$ as follows: First, define $h(n)=\sup _{k\geq n}M_k$. Then $h(n+1)\leq h(n)$, and by Lemma
\ref{ConvergenceOfTildeV_TildeMProp}, we see that $\lim _{n\rightarrow\infty} h(n)=0$. Then we extend it appropriately.

We take $\epsilon _0=\displaystyle{\frac{h(1)}{2}}$. Since $\displaystyle{\frac{h(x)}{x+1}}$ is continuous and strictly decreasing, and $\lim _{x\rightarrow\infty}h(x)=0$, we can define a function $\varphi :(0,\epsilon _0)\rightarrow (0,\infty)$ as follows:
\begin{equation}
\varphi (\epsilon )=h(x)\mbox{ iff }\epsilon =\frac{h(x)}{x+1}.\label{EquationOfPhi}
\end{equation}

We note that $\varphi $ is non-decreasing and $\lim _{\epsilon \rightarrow 0}\varphi (\epsilon)=0$.
\section{Proof of Theorem \ref{TheoremMain}}\label{SectionProofOfMainTheorem}
Fix $R>0$. Let $g_p(E,\epsilon ,R)$ be as in (\ref{EquationOfFunctionG}), let $\varepsilon _0$ and $\varphi (\epsilon )$ be as in the previous section.
Let $\alpha >0$ be the constant in Lemma \ref{LemmaExistenceOfAlpha}.
\begin{proof}{(of Theorem \ref{TheoremMain})}
By definition of $C_p(E,\epsilon ,R)$ and $g(E,\epsilon ,R,q)$, recall that $||q||_{\infty}=1$ it follows that $g(E,\epsilon ,R,q)\leq C_p(E,\epsilon
,R)$. Hence it remains to prove the right hand-sided inequality of (\ref{EquationEstimateCaseE0HasZeroMeasure}).

Let $Z_n=(z_1,\ldots ,z_n)\in E^n$. It follows from Theorem \ref{TheoremRecovery} that
 \begin{equation}
C_p(E,\epsilon ,R)\leq K\times \varepsilon \mu (Z_n)\times \max_{1\leq k\leq n}\sup _{|z|\leq R}|B(Z_{n,k},z)|+K\sup _{|z|\leq
R}|B(Z_n,z)|,\label{EqautionProofOfMainTheorem.1}
\end{equation}
for some constant $K>0$ depending only on $R$ and $p$. Applying Lemma \ref{LemmaExistenceOfAlpha}, remark that $0<\alpha <1$, we obtain
\begin{equation}
C_p(E,\epsilon ,R)\leq K\times |q(0)|^{-\alpha}\times (\varepsilon \mu (Z_n)+1)\times \sup _{|z|\leq
R}|B_q(Z_n,z)|^{\alpha}.\label{EquationTheoremMain.1}
\end{equation}

It follows from Proposition \ref{ConvergenceOfTildeV_TildeMProp} that $\displaystyle\lim_{n \rightarrow \infty}  M_n(E)$ $=$ $0$. Thus, we can choose the
smallest $n_0$ such that $ M_{n_0}(E)$ $\leq$ $\varphi (\varepsilon )$ $<$ $ M_{n_0-1}(E)$ for all $\varepsilon$ less than $\varepsilon_0$. Then, by
Proposition \ref{Mu_TildeMRelationProp}
\[
 M_{n_0}(E) \leq \varphi (\varepsilon ) <  M_{n_0-1}(E) \leq \frac{n_0}{\mu (Z_{n_0})},
\]
for any finite sequence $Z_{n_0}=\{z_1,\ldots ,z_{n_0}\}\in E^{n_0}$ with $V(Z_{n_0})=V_{n_0}(E),~M(Z_{n_0})=M_{n_0}(E)$. In particular, for such
$Z_{n_0}$ we have $\varphi (\varepsilon )\mu (Z_{n_0})\leq n_0$.

On the other hand, we have $\varphi (\epsilon ) < M_{n_0-1}(E)$ $\leq$ $h(n_0-1)$ for $n_0$ $\geq$ $N$. This and (\ref{EquationOfPhi}) give $n_0$ $\leq$
$x+1$, where
\begin{eqnarray*}
\epsilon =\frac{h(x)}{x+1}.
\end{eqnarray*}
 Hence,
\begin{equation}
\epsilon \mu (Z_{n_0})=\frac{\epsilon}{\varphi (\epsilon )}\varphi (\epsilon )\mu (Z_{n_0})\leq\frac{\epsilon}{\varphi (\epsilon )}n_0
\leq\frac{\epsilon}{\varphi (\epsilon )}(x+1)=\frac{\frac{h(x)}{x+1}}{h(x)}(x+1)=1.\label{EqautionProofOfMainTheorem.2}
\end{equation}

Now, $M(Z_{n_0})= M_{n_0}(E)\leq \varphi (\varepsilon )$ implies that
\begin{eqnarray*}
\sup _{|z|\leq R}|B_q(Z_{n_0},z)|\leq g(E,\varphi (\epsilon ),R,q).
\end{eqnarray*}
This, together with (\ref{EqautionProofOfMainTheorem.2}), plugged into (\ref{EquationTheoremMain.1}) yields
\begin{eqnarray*}
C_p(E,\epsilon ,R)\leq 2K \times |q(0)|^{-\alpha}\times g^{\alpha}(E,\varphi (\epsilon ),R,q).
\end{eqnarray*}
This concludes the proof of Theorem \ref{TheoremMain}.
\end{proof}
\section{Corollaries and examples}\label{SectionCorollaryAndExample}
We keep the same assumptions as in Section \ref{SectionProofOfMainTheorem}.
\begin{corollary}
If there exist $C>0$, $\sigma >0$ and $N>0$ such that $M_n(E)\leq Cn^{-\sigma}$ for all $n\geq N$ then there exists $\epsilon _0>0$ depending only on $E$
and there exists $\kappa >0$ depending only on $\sigma$ such that
\begin{equation}
g(E,\epsilon ,R,q)\leq C_p(E,\epsilon ,R)\leq K\times |q(0)|^{-\alpha}\times g^{\alpha}(E,\epsilon ^{\sigma },R,q).\label{EquationEstimateCaseMnIsSmall}
\end{equation} \label{CorollaryMnIsSmall}\end{corollary}
\begin{proof}
If $M_n (E)\leq Cn^{-\sigma }$ for all $n\geq N$ then we choose $h(x)=Cx^{-\sigma}$. So we have
\begin{eqnarray*}
\varphi (\epsilon)=h(x)=Cx^{-\sigma}\leq C_2\epsilon ^{\sigma /(1+\sigma )},
\end{eqnarray*}
since $\epsilon =Cx^{-\sigma}(1+x)^{-1}$. Applying Theorem \ref{TheoremMain} completes the proof of Corollary \ref{CorollaryMnIsSmall}.
\end{proof}
\begin{proof}(Of Corollary \ref{CorollaryCaseEIsCompact})
Since $\overline E\subset U$, there exists $0<r<1$ such that $\sup _{z\in E}|z|\leq r$. We can also choose $q\equiv 1$. Hence we get that $M_n(E)\leq
\displaystyle{(\frac{2r}{1+r})^n}$. So the function $\varphi (\epsilon )$ in Theorem \ref{TheoremMain} satisfies
\begin{equation}
\lim _{\epsilon\rightarrow 0}\frac{\log \epsilon}{\log \varphi (\epsilon )}=1.\label{LRS.2}
\end{equation}
Applying Theorem \ref{TheoremMain} completes the proof of Corollary \ref{CorollaryCaseEIsCompact}.
\end{proof}
Corollary \ref{CorollaryCaseEIsInStolzAngles} is a consequence of Corollary \ref{CorollaryMnIsSmall}, because of the following result
\begin{proposition}
Assume that $E$ is contained in some Stolz angles. Then there exist $\sigma$, $C$ and $N$ $>$ $0$ such that $ M_n(E)$ $\leq$ $Cn^{-\sigma}$ for $n$ $\geq$ $N$.
\label{CorrectnessOfGeneralizationProp}
\end{proposition}

\begin{proof}{}

Let $\overline{E}\cap \partial U=\{a_1,a_2,...,a_n\}$. We take in this case $q(z)=(z-a_1)(z-a_2)...(z-a_n)$.

We seperate the proof into three steps.
\begin{enumerate}[1.]
\item Suppose that $\overline E$ lies inside U. In this case $M_n$ $\leq$ $n^{-\sigma}$ for sufficiently large $n$ (see Corollary \ref{CorollaryCaseEIsCompact}).

\item Suppose that $\overline E$ $\cap$ $\partial U$ has only one point. By means of some rotation, we may assume that it this point is $1$.

We have $q(z)$ $=$ $z - 1$. We see that if $|q(z)|$ $>$ $\delta$ $>$ $0$ for some $z$ in $E$, then $|z|$ $<$ $r_\delta$ = $1 - c\delta$ where $c$ is a constant depending on the Stolz angle with vertex at $1$. Refering to the proof of Proposition \ref{ConvergenceOfTildeV_TildeMProp}, we get
\begin{equation}
 V_n^{1/n} \leq C\max\{\delta^{1/3}, \eta^{n/9}\}.
\label{eqmax}
\end{equation}
Choosing $\delta$ $=$ $n^{-3\sigma}$ ($\sigma$ $\in$ $(0,1/6)$), we have
\[
\eta = \frac{2r_\delta}{1 + r_\delta^2} = \frac{2(1 - c n^{-3\sigma})}{1 + (1 - c n^{-3\sigma})^2} = \frac{2n^{6\sigma} - 2cn^{3\sigma}}{2n^{6\sigma} - 2cn^{3\sigma} + c^2}.
\]
Hence,
\begin{align*}
\eta^{n/9}
    &= \left(\frac{2n^{6\sigma} - 2cn^{3\sigma}}{2n^{6\sigma} - 2cn^{3\sigma} + c^2}\right)^{n/9} = \left(1 - \frac{c^2}{2n^{6\sigma} - 2cn^{3\sigma} + c^2}\right)^{n/9}\\
    &\leq \left(1 - \frac{c^2}{2n^{6\sigma}}\right)^{\frac{2n^{6\sigma}}{c^2}\frac{c^2n^{1-6\sigma}}{18}} \leq \exp\left(-\frac{c^2n^{1-6\sigma}}{18}\right) \leq n^{-\sigma}
\end{align*}
for sufficiently large $n$. Combining with (\ref{eqmax}), the assertion follows.

\item Now, consider the general case. It suffices to show that if $E_1$ and $E_2$ are two sets satisfy $ V_n^{1/n}(E_i)$ $\leq$ $C n^{-\sigma_i}$ for $n$ $\geq$ $N$ ($i$ $=$ $1$, $2$) and $E$ $=$ $E_1$ $\cup$ $E_2$, then $ V_n^{1/n}(E)$ $\leq$ $C n^{-\sigma}$, for $n$ $\geq$ $2N$ and $\sigma$ $=$ $\min\{\sigma_1, \sigma_2\}/2$. We take $q(z)=q_1(z)q_2(z)$where $q_1,~q_2$ are coressponding $q's$ functions of $E_1,~E_2$.  Fix an $n$ $\geq$ $2N$ and suppose that $ V_n(E)$ $=$ $ V(z_1, z_2, \ldots, z_l, \zeta_1, \zeta_2, \ldots, \zeta_k)$ for $z_j$ $\in$ $E_1$, $\zeta_j$ $\in$ $E_2$, and $n$ $=$ $l + k$. It follows from definitions that
\[
 V_n^{1/n}(E) \leq  V_l^{1/n}(E_1)  V_k^{1/n}(E_2).
\]
We may assume that $l$ $\geq$ $k$. It implies that $l$ $\geq$ $n/2$ $\geq$ $N$. If $k$ $\leq$ $N$, we have
\[
 V_n^{1/n}(E) \leq C V_l^{1/n}(E_1) \leq C l^{-\sigma_1 l/n} \leq C (n/2)^{- \sigma_1/2} \leq C n^{-\sigma}.
\]
If $k$ $\geq$ $N$, we have
\begin{align*}
 V_n^{1/n}(E)
    &\leq  V_l^{1/n}(E_1)  V_k^{1/n}(E_2) \leq C l^{-\sigma_1 l/n}k^{-\sigma_2 k/n} \leq C \left(l^{-l/n}k^{-k/n}\right)^{-\sigma}\\
    &= C n^{-\sigma}\left((l/n)^{-l/n}(k/n)^{-k/n}\right)^{-\sigma} \leq C n^{-\sigma}.
\end{align*}
Here we have used the inequality $x^x(1-x)^{1-x}$ $\geq$ $1/2$ for all $x$ $\in$ $(0,1)$.
\end{enumerate}
The proof is complete.
\end{proof}

We conclude this section providing more sets $E$ satisfying the condition of Corollary \ref{CorollaryMnIsSmall}. For convenience, we recall some
definitions that Hayman used in constructing the function $f$ in Theorem \ref{TheoremHayman}.

\begin{definition} Let $E$ satisfy $(G)$. We write
\begin{eqnarray*}
E'=\{z=re^{i\theta}:~|\theta -\phi |< 1-r\mbox{ and }re^{i\phi} \in E \}.
\end{eqnarray*}

Next, for $0\leq\theta\leq 2\pi$, we define
\begin{eqnarray*}
\rho (\theta )=\sup \{\rho :~0\leq \rho <1,~\rho e^{i\theta}\in E'\}.
\end{eqnarray*}

Let $E_{\infty}$ be the set of $\theta$ such that $\rho (\theta )=1$. If $\theta \in E_\infty$ then $e^{i\theta}\in E_0$. So $m(E_\infty )=0$, where $m(.)$ is the Lebesgue's measure of the unit circle.

For each $1>r>0$ let $E_r$ be the set of all $\theta$ such that $0\leq \theta \leq 2\pi$ and $\rho (\theta )> r$. Then $E_r$ are open and contract with increasing $r$, and
\begin{eqnarray*}
\bigcap _{r}E_r=E_{\infty}.
\end{eqnarray*}
Thus
\begin{eqnarray*}
\lim _{r\rightarrow 1}m(E_r)=0.
\end{eqnarray*}
\end{definition}

Considering carefully the construction in the proof of Theorem 1 in \cite{hay} and Step 2 of the proof of Proposition \ref{CorrectnessOfGeneralizationProp} we can show that if the quantities $m(E_r)$ tend to $0$ sufficiently fast, then $M_n\leq Cn^{-\sigma}$. In particular, this claim is true if the following condition is satisfied
\begin{eqnarray*}
m(E_{\delta})\leq \frac{1}{-2\log \epsilon} \mbox{ if } \delta =1-K\sqrt{-\epsilon\log \epsilon},
\end{eqnarray*}
where $K$ is a positive constant. In fact, if this condition holds, the function $f$ is constructed in Theorem 1 in \cite{hay} will satisfy: if
$|f(z)|>\epsilon$ then $|z|\leq 1-K\sqrt{-\epsilon\log \epsilon}$. This last inequality ensures that $E$ satisfies conditions of Corollary
\ref{CorollaryMnIsSmall} (see proof of Proposition \ref{ConvergenceOfTildeV_TildeMProp}).
\section{One-point estimate}\label{SectionOnePointEstimate}
In this section we sketch how to obtain similar results for the case of one-point estimate, that is of estimating $C_p(E,\epsilon ,0)$ in
(\ref{Ce0definition}). There are two cases:

Case 1: $0\in E$. In this case it is easy to see that $C_p(E,\epsilon ,0)=\epsilon$.

Case 2: $0\not\in E$. In this case there exists $1>r>0$ such that if $z\in E$ then $|z|\geq r$. Then we can define similar set functions like those in
Section \ref{SectionSetFunctions} to obtain similar result to that of Theorem \ref{TheoremMain} and Corollaries \ref{CorollaryCaseEIsCompact},
\ref{CorollaryCaseEIsInStolzAngles} and \ref{CorollaryMnIsSmall}.
\section{Appendix: Case $m(E_0)>0$}\label{SectionAppendix}
In this section we present the proof of (\ref{EquationEstimateCaseE0HasPositiveMeasure}) when $m(E_0)>0$. We thank Professor Yuril Lyubarskii for showing
us this proof.
\begin{proof} (Of (\ref{EquationEstimateCaseE0HasPositiveMeasure}))

Since $m(E_0)>0$ the harmonic measure $\omega (z)$ of $E_0$ (see \cite{ransford}) satisfies: $\omega$ is a harmonic function in $\mathbb{D}$, $0<\omega
(z)<1$ for all $z\in \mathbb{D}$, (its boundary value) $\omega (z)=1$ a.e for $z\in E_0$, $\omega (z)=0$ for a.e $z\in
\partial \mathbb{D}\backslash E_0$. Let $v(z)$ be an analytic function with real part $\omega$.

For any $\varepsilon >0$ define
\begin{eqnarray*}
u_{\varepsilon }(z)=\exp \{\log \varepsilon \times v(z)\}=\varepsilon ^{v(z)}.
\end{eqnarray*}
Then $u_{\varepsilon}$ is analytic in $\mathbb{D}$, $0<|u_{\varepsilon }(z)|=\varepsilon ^{\omega (z)}<1$ for all $z\in \mathbb{D}$, $|u_{\varepsilon
}(z)|=\varepsilon$ a.e for $z\in E_0$, $|u_{\varepsilon }(z)|=1$ for a.e $z\in
\partial \mathbb{D}\backslash E_0$.

Let $f$ be any function in $\mathcal{A}^p$ with $|f(z)|\leq \varepsilon$ for all $z\in E$. Then $|f(z)|\leq \varepsilon$ a.e in $E_0$. Then
$f/u_{\varepsilon}$ is holomorphic in $\mathbb{D}$ and we have

\begin{eqnarray*}
\frac{1}{2\pi}\int _0^{2\pi}|\frac{f}{u_{\varepsilon }}(e^{it})|^pdt&=&\frac{1}{2\pi}\int _{t \in E_0}\frac{|f(e^{it})|^p}{|\varepsilon
|^p}dt+\frac{1}{2\pi}\int
_{t\not\in E_0}\frac{|f(e^{it})|^p}{1}dt\\
&\leq &\frac{1}{2\pi}\int _{t \in E_0}dt +\frac{1}{2\pi}\int _{0}^{2\pi}|f(e^{it})|^pdt\leq 2.
\end{eqnarray*}
Hence $||f/u_{\varepsilon}||_{H^p}\leq 2^{1/p}$. Applying (\ref{EquationEstimateForFunctionsInSpaceAp}) to $f/u_{\varepsilon}$ and use the definition of
$u_{\varepsilon}$ we obtain (\ref{EquationEstimateCaseE0HasPositiveMeasure}).
\end{proof}


\begin{thebibliography}{1}
\bibitem{dan} N. Danikas, {\it On an identity theorem in Nevalina class}, J. Appros. Theory 77 (1994), 184-190.

\bibitem{hay} W. K. Hayman, {\it Identity theorems for functions of bounded characteristic}, J. London Math. Soc.,(2) 58 (1998), 127-140.

\bibitem{LRS} M. M. Lavrent'ev, V. G. Romanov, S. P. Shishat-skii, {\it Ill-posed problems of mathematical Physics and Analysis}, Transl. Amer. Math. Soc. (1986).

\bibitem{lk} Yuril I. Lyubarskii and Kristian Seip, {\it A uniqueness theorem for bounded analytic functions}, Bull. London Math. Soc.,29 (1997), 49-52.

\bibitem{osi} K. Yu. Osipenko, {\it Best approximation of analysis functions from information about their values at a finite number of points}, Math. Zametki,19 (1976), 29-40.

\bibitem{osi7} K. Yu. Osipenko, {\it TheHeins problem and optimal extrapolation of analytic functions given with error}, Mat. Sb. (N.S.) 126(168) (1985), no. 4, 566--575.

\bibitem{osi6} K. Yu. Osipenko, {\it Best and optimal quadrature formulas on classes of bounded analytic functions}, Izv. Akad. Nauk SSSR Ser. Mat. 52 (1988), no. 1, 79--99, 240;
translation in Math. USSR-Izv. 32 (1989), no. 1, 77--97.

\bibitem{osi5} K. Yu. Osipenko, {\it Blaschke products that deviate least from zero}, Mat. Zametki 47 (1990), no. 5, 71--80, 159; translation in Math. Notes 47 (1990),
no. 5-6, 471--477.

\bibitem{osi4} K. Yu. Osipenko and M. I. Stessin, {\it On some problems of optimal recovery of analytic and harmonic functions from inaccurate
data}, J. Approx. Theory 70 (1992), no. 2, 206--228.

\bibitem{osi3} K. Yu. Osipenko, {\it Optimal reconstruction of analytic functions from their values in a uniform grid on a circle}, Vladikavkaz. Mat. Zh. 5 (2003),
no. 1, 48--52 (electronic).

\bibitem{osi2} K. Yu. Osipenko, {\it TheHardy-Littlewood-Polya inequality for analytic functions in Hardy-Sobolev spaces}, Mat. Sb. 197 (2006), no. 3, 15--34;
translation in Sb. Math. 197 (2006), no. 3-4, 315--334.

\bibitem{ransford}{Thomas Ransford,} \textit{Potential theory in the complex plane,} London Maths. Soc. Student Texts 28, Cambridge University Press, 1995.

\bibitem{rud}{ Walter Rudin,} \textit{Real and complex analysis,} International Student Edition, McGraw-Hill, London - New York -et al., Mladinska Knjiga: Ljubljana, 1970.

\bibitem{TLNT}D. D. Trong, N. L. Luc, L. Q. Nam, T. T. Tuyen, {\it Reconstruction of $H^p$-functions: Best approximation, regularization and optimal error estimate}, Complex Variables, Vol 49, no.4, 2004, 285-301.

\bibitem{Tsuji}M. Tsuji, {\it Potential theory in modern function theory}, Maruzen, 1959.
\end{thebibliography}
\end{document}